\documentclass[a4paper,12pt]{article}
\usepackage{cmap}
\usepackage[cp1251]{inputenc}
\usepackage[english]{babel}
\usepackage[left=2cm,right=2cm,top=2cm,bottom=2cm]{geometry}
\usepackage{amssymb}
\usepackage{amsmath, amsthm}
\theoremstyle{plain}
\newtheorem{theorem}{Theorem}
\newtheorem{lemma}{Lemma}
\newtheorem{proposition}{Proposition}
\newtheorem{corollary}{Corollary}

\theoremstyle{definition}

\newtheorem{definition}{Definition}

\begin{document}

\begin{center}\Large
\textbf{On some applications of \\
Fitting like subgroups of finite groups}
\normalsize

\smallskip
Viachaslau I. Murashka and Alexander F. Vasil'ev

 \{mvimath@yandex.ru, formation56@mail.ru\}

Faculty of Mathematics and Technologies of Programming,

 Francisk Skorina Gomel State University,  Gomel 246019, Belarus\end{center}

\begin{abstract}
  In this paper we study the groups all whose maximal or all Sylow subgroups are $K$-$\mathfrak{F}$-subnormal in their product the with  generalizations $\mathrm{F}^*(G)$ and $\mathrm{\tilde F}(G)$ of the Fitting subgroup. We prove that  a hereditary formation $\mathfrak{F}$  contains every group all whose Sylow subgroups are $K$-$\mathfrak{F}$-subnormal in their product with $\mathrm{F}^*(G)$ if and only if $\mathfrak{F}$ is the class of all $\sigma$-nilpotent groups for some partition $\sigma$ of the set of all primes.
 We obtain a new characterization of the  $\sigma$-nilpotent hypercenter, i.e. the $\mathfrak{F}$-hypercenter and the largest normal  subgroup which $K$-$\mathfrak{F}$-subnormalize all Sylow subgroups coincide if and only if $\mathfrak{F}$ is the class of all $\sigma$-nilpotent groups.
\end{abstract}


\textbf{AMS}(2010). 20D25,  20F17,   20F19.

\section*{Introduction}

Throughout this paper, all groups are finite,   $G$, $p$ and  $\mathfrak{X}$  always denote a finite group,     a prime  and a class of groups, respectively.

Recall that  every group $G$ has the largest normal nilpotent subgroup
 $\mathrm{F}(G)$. This subgroup is called \emph{the Fitting subgroup}.  \emph{The generalized Fitting subgroup} $\mathrm{F}^*(G)$ was introduced by  H.~Bender \cite{f2} and  can
 \cite[X, Theorem 13.13]{19} be defined by
$$ \mathrm{F}^*(G)/\mathrm{F}(G)=\mathrm{Soc}(\mathrm{F}(G)C_G(\mathrm{F}(G))/\mathrm{F}(G)).$$
Another generalization $\tilde{\mathrm{F}}(G)$ of the Fitting subgroup    was introduced by P. Schmid \cite{f3} and L.A.~Shemetkov
 \cite[Definition 7.5]{f4}. This subgroup is defined by
 $$ \Phi(G)\subseteq \tilde{\mathrm{F}}(G) \textrm{ and
  } \tilde{\mathrm{F}}(G)/\Phi(G)=\mathrm{Soc}(G/\Phi(G)).$$
  P. F\"orster \cite{F2} showed that $\tilde{\mathrm{F}}(G)$ can also be defined by
 $ \tilde{\mathrm{F}}(G)/\Phi(G)=\mathrm{F}^*(G/\Phi(G)).$
Note that $\mathrm{F}(G)=\mathrm{F}^*(G)=\mathrm{\tilde F}(G)$ for a soluble group $G$.  It is well known that $C_G(\mathrm{F}(G))\subseteq\mathrm{F}(G)$  for every soluble group $G$. In the universe of all groups $\mathrm{F}(G)$ does not have this property but $\mathrm{F}^*(G)$ and $\tilde{\mathrm{F}}(G)$  have.

Let $\mathfrak{F}$ be a formation.  Recall \cite[Definition 6.1.4]{s9} that a subgroup $H$ of  $G$ is called $K$-$\mathfrak{F}$-\emph{subnormal} in $G$ if there is a chain
$ H=H_0\subseteq H_1\subseteq\dots\subseteq H_n=G$
with $H_{i-1}\trianglelefteq H_i$ or $H_{i}/\mathrm{Core}_{H_{i}}(H_{i-1})\in\mathfrak{F}$ for all $i=1,\dots,n$. Denoted by $H\,K$-$\mathfrak{F}$-$sn\,G$.

If $\mathfrak{F}=\mathfrak{N}$ is the formation of all nilpotent groups, then the notions of  $K$-$\mathfrak{F}$-subnormal and subnormal subgroups coincide. Groups with different systems of $K$-$\mathfrak{F}$-subnormal are the main object of many papers (for example, see \cite{vF, gs3, wF,sp4}).
The main idea of this paper is to consider  $K$-$\mathfrak{F}$-subnormality of a subgroup not in the whole group but in some subgroup related to some generalization of the Fitting subgroup in the sense of the following definition:

\begin{definition}\label{de1}
   Let $\mathfrak{F}$ be a formation and $R$ be a subgroup of a group  $G$. We shall call a subgroup   $H$ of $G$  $R$-$K$-$\mathfrak{F}$-\emph{subnormal} if $H$ is $K$-$\mathfrak{F}$-subnormal in $\langle H, R\rangle$.
 \end{definition}

If $\mathfrak{F}=\mathfrak{N}$, then we just obtain the notion of   $R$-subnormal subgroup. In \cite{MonKon, MonChir, mv2, mv1, mv3} the products of $R$-subnormal subgroups were studied for   $R\in\{\mathrm{F}(G), \mathrm{F}^*(G)\}$. It was shown that if  $G$ is the product of two nilpotent (resp. quasinilpotent) $\mathrm{F}(G)$-subnormal (resp. $\mathrm{F}^*(G)$-subnormal) subgroups, then it is nilpotent (resp. quasinilpotent).

It is well known that a group is nilpotent if all its maximal or all its Sylow subgroups are subnormal. In  \cite{Arxiv1} it was shown that we can replace subnormality in this result by $\tilde{\mathrm{F}}(G)$-subnormality for maximal subgroups and  $\mathrm{F}^*(G)$-subnormality for Sylow subgroups. The main results of this paper are

\begin{theorem}\label{t1.0}
  Let   $\mathfrak{F}$ be a formation. Then

  $(1)$ $\mathfrak{F}$ contains every group $G$ all whose maximal subgroups are  $\tilde{\mathrm{F}}(G)$-$K$-$\mathfrak{F}$-subnormal iff $\mathfrak{F}$ is saturated and contains $\mathfrak{N}$.

  $(2)$ $\mathfrak{F}$ contains every group $G$ all whose maximal subgroups are  $\mathrm{F}^*(G)$-$K$-$\mathfrak{F}$-subnormal iff $\mathfrak{F}$ is the class of all groups.
\end{theorem}

\begin{corollary}[\cite{Arxiv1}]
A group $G$ is nilpotent if and only if every its maximal subgroup is  $\tilde{\mathrm{F}}(G)$-subnormal.
\end{corollary}

 %Известный результат Крамера [4, c. 12] можно переформулировать следующим
%образом.

\begin{corollary}[Kramer \cite{f5}]\label{kr}
   A soluble group $G$ is supersoluble if and only if   $\mathrm{F}(G)\leq M$ or $M\cap\mathrm{F}(G)$ is a maximal subgroup of $\mathrm{F}(G)$ for every maximal subgroup $M$ of $G$. \end{corollary}

In the paper \cite{30} another generalization of Kramer's result was obtained with the help of $\tilde{\mathrm{F}}(G)$. In the spirit of Definition \ref{de1} their result can be reformulated in the following way: \flqq A group $G$ is supersoluble if and only if every its maximal subgroup is $\tilde{\mathrm{F}}(G)$-$\mathbb{P}$-subnormal\frqq. Note that $\mathbb{P}$-subnormality and $K$-$\mathfrak{U}$-subnormality coincide in the class of all soluble groups. In the general case these notions are different \cite{gs3}.

Let $\sigma=\{\pi_i\,|\,i\in I\}$ be a partition of the  set $\mathbb{P}$ of all primes. According to A.N. Skiba \cite{sp4}, a group $G$ is called $\sigma$-\emph{nilpotent}  if  $G$ has a normal Hall $\pi_i$-subgroup for every $i\in I$ with $\pi(G)\cap\pi_i\neq\emptyset$. The class of all $\sigma$-nilpotent groups is denoted by  $\mathfrak{N}_\sigma$. This class is a very interesting generalization of the class of nilpotent groups and widely studied.

   \begin{theorem}\label{t1.1} Let  $\mathfrak{F}$ be a hereditary formation. The following statements are equivalent:

  $(1)$ $\mathfrak{F}$ contains every group  $G$ all whose cyclic primary subgroups are  $\mathrm{F}^*(G)$-$K$-$\mathfrak{F}$-subnormal.

 $(2)$  $\mathfrak{F}$ contains every group  $G$ all whose Sylow subgroups are  $\mathrm{F}^*(G)$-$K$-$\mathfrak{F}$-subnormal.

 $(3)$ There is a partition   $\sigma$ of $\mathbb{P}$ such that $\mathfrak{F}$ is the class of all  $\sigma$-nilpotent groups.      \end{theorem}

   The next result follows from the previous theorem.

   \begin{theorem}\label{new} Let  $\mathfrak{F}$ be a hereditary formation. The following statements are equivalent:

  $(1)$ $\mathfrak{F}$ contains every group $G=AB$ where all cyclic primary subgroups of $A$ and $B$ are  $\mathrm{F}^*(G)$-$K$-$\mathfrak{F}$-subnormal.

 $(2)$ $\mathfrak{F}$ contains every group $G=AB$ where all Sylow subgroups of $A$ and $B$ are  $\mathrm{F}^*(G)$-$K$-$\mathfrak{F}$-subnormal.

 $(3)$ There is a partition   $\sigma$ of $\mathbb{P}$ such that $\mathfrak{F}$ is the class of all  $\sigma$-nilpotent groups.      \end{theorem}

  Recall that a subgroup $H$ of  $G$ is called  $R$-conjugate-permutable
\cite{31} if $H^rH=HH^r$ for all $r\in R$. If $R=G$, then we obtain the notion of conjugate-permutable subgroup \cite{7}. From $(1)$ of  \cite[Lemma 2.2]{31} it follows that    $\mathrm{F}^*(G)$-conjugate-permutable subgroup is $\mathrm{F}^*(G)$-$K$-$\mathfrak{N}$-subnormal. Hence the main result  of \cite{kit} follows from Theorem   \ref{new}.

\begin{corollary} Let   $A$ and $B$ be a subgroups of a group   $G$ and $G = AB$. If every Sylow subgroup of   $A$  is $B\mathrm{F}^*(G)$-conjugate-permutable and every Sylow subgroup of  $B$   is $A\mathrm{F}^*(G)$-conjugate-permutable, then $G$ is nilpotent.\end{corollary}

\section{Preliminaries}

The notation and terminology agree with \cite{s9, s8}. We refer the reader to these
books for the results about formations.

Recall that a \emph{formation} is a class of groups which is closed under taking epimorphic images and subdirect products. A formation $\mathfrak{F}$ is said to be: \emph{saturated}  if $G\in\mathfrak{F}$
whenever $G/\Phi(G)\in\mathfrak{F}$; \emph{hereditary} if $H\in \mathfrak{F}$ whenever $H\leq G\in \mathfrak{F}$. The following two lemmas follow from \cite[Lemmas 6.1.6 and 6.1.7]{s9}.

\begin{lemma}\label{l3.1}  Let $\mathfrak{F}$ be a formation, $H$ and  $R$ be subgroups of a group  $G$ and $N\trianglelefteq G$.

$(1)$     If $H$ $K$-$\mathfrak{F}$-$sn\,G$, then $HN/N$ $K$-$\mathfrak{H}$-$sn\,G/N$.

$(2)$ If  $H/N $~$K$-$\mathfrak{F}$-$sn\,G/N$, then   $H$~$K$-$\mathfrak{F}$-$sn\,G$.

%(3) Если $H$ $K$-$\mathfrak{F}$-$sn\,G$, то $HN$ $K$-$\mathfrak{F}$-$sn\,G$.

$(3)$   If $H $ $K$-$\mathfrak{F}$-$sn\,R$  and  $R $~$K$-$\mathfrak{H}$-$sn\,G$, then $H$ $K$-$\mathfrak{F}$-$sn\,G$.

%(5) {\sl Пусть $G$ --- $p$-группа, где $p\in\mathbb{P}$. Если $Z_p\in\mathfrak{H}$, то все подгруппы $G$ являются \linebreak$\mathfrak{H}$-субнормальными}.
\end{lemma}

\begin{lemma}\label{l3.2}
 Let $\mathfrak{F}$ be a hereditary formation, $H$ and  $R$ be subgroups of a group $G$.

$(1)$  If   $H$~$K$-$\mathfrak{F}$-$sn\,G$, then  $H\cap R $~$K$-$\mathfrak{F}$-$sn\,R$.

$(2)$ If  $H $\,$K$-$\mathfrak{F}$-$sn\,G$ and $R $\,$K$-$\mathfrak{F}$-$sn\,G$, then   $H\cap R $\,$K$-$\mathfrak{F}$-$sn\,G$.
\end{lemma}

The following lemma directly follows from Lemma \ref{l3.1}.

\begin{lemma}\label{lemN}
  Let $\mathfrak{F}$ be a formation, $H$ and  $R$ be  subgroups of a group  $G$ and $N\trianglelefteq G$. If $H$ $K$-$\mathfrak{F}$-$sn\,R$, then $HN$ $K$-$\mathfrak{F}$-$sn\,RN$.
\end{lemma}

The following result directly follows from \cite[B, Theorem 10.3]{s8}.

\begin{lemma}\label{10.3B}
  Let $p$ be a prime and  $G$ be a group. If $\mathrm{O}_p(G)=1$ and $G$ has a unique minimal normal subgroup, then there exists a faithful irreducible  $\mathbb{F}_pG$-module.
\end{lemma}

Recall \cite[Chapter 6.3]{s9} or \cite{VKS} that  a formation $\mathfrak{F}$ has \emph{the lattice property for $K$-$\mathfrak{F}$-subnormal subgroups} if the set of all $K$-$\mathfrak{F}$-subnormal subgroups is
a sublattice of the lattice of all subgroups in every group.

\begin{lemma}[see \cite{VKS}, {\cite[Lemma 2.6(3)]{sp4}}]\label{lattice}
  Let $\sigma$ be a partition of $\mathbb{P}$.  $\mathfrak{N}_\sigma$ has the lattice property for $K$-$\mathfrak{N}_\sigma$-subnormal subgroups.
\end{lemma}

\section{The main steps of the proof of  Theorem  \ref{t1.1}}

   %\textbf{??Questions. About comosition and $\mathrm{\tilde{F}}(G)$-$\mathbb{P}$-$\mathfrak{F}$-subnormal.}

   The proof of  Theorem  \ref{t1.1} is rather complicated and require  various  preliminary results and definitions.

   Let $\mathfrak{X}$ be a class of groups.  A chief factor $H/K$ of  $G$ is called   $\mathfrak{X}$-\emph{central} in $G$ provided    $(H/K)\rtimes (G/C_G(H/K))\in\mathfrak{X}$ (see \cite[p. 127--128]{s6}). A normal subgroup $N$ of $G$ is said to be $\mathfrak{X}$-\emph{hypercentral} in $G$ if $N=1$ or $N\neq 1$ and every chief factor of $G$ below $N$ is $\mathfrak{X}$-central. The symbol $\mathrm{Z}_\mathfrak{X}(G)$ denotes the $\mathfrak{X}$-\emph{hypercenter} of $G$, that is, the largest normal $\mathfrak{X}$-hypercentral subgroup of $G$ (see  \cite[Lemma 14.1]{s6}). If $\mathfrak{X}=\mathfrak{N}$, then $\mathrm{Z}_\mathfrak{N}(G)$ is the hypercenter of $G$.

   Let $H$ be  a subgroup of a group $G$. According to \cite[p. 50]{s8}  a subgroup  $T$ of $G$ is called  \emph{a subnormalizer} of $H$ in $G$ if $H$ is subnormal in $T$ and if $H$ is subnormal in $M\leq G$, then $M\leq T$. A subnormalizer, if it exists, is unique.    A subgroup  $T$ of $G$ is called \emph{a weak subnormalizer} of $H$ in $G$ \cite{h9} if $H$ is subnormal in $T$ and if $H$ is subnormal in $M\leq G$ and $T\leq M$, then $T=M$. A weak subnormalizer always exists but may be not  unique. We introduce the following generalization of the previous concept:

   \begin{definition}
    Let $\mathfrak{F}$ be a formation.  We shall call a subgroup  $T$ of $G$   \emph{a weak $K$-$\mathfrak{F}$-subnormalizer} of $H$ in $G$  if $H$ is $K$-$\mathfrak{F}$-subnormal in $T$ and if $H$ is $K$-$\mathfrak{F}$-subnormal in $M\leq G$ and $T\leq M$, then $T=M$.
   \end{definition}

   The following result plays the main role in the proof of Theorem \ref{t1.1}.

   \begin{theorem}\label{gb}  Let  $\mathfrak{F}$ be a hereditary formation. The following statements are equivalent:

  $(1)$ The intersection of all weak  $K$-$\mathfrak{F}$-subnormalizers of all cyclic primary subgroups is the  $\mathfrak{F}$-hypercenter.

 $(2)$  The intersection of all weak  $K$-$\mathfrak{F}$-subnormalizers of all Sylow subgroups is the  $\mathfrak{F}$-hypercenter.

 $(3)$ There is a partition   $\sigma$ of $\mathbb{P}$ such that $\mathfrak{F}$ is the class of all  $\sigma$-nilpotent groups.     \end{theorem}

\begin{corollary}[P. Hall \cite{h2}] The intersection of all normalizers of Sylow subgroups is the hypercenter.
\end{corollary}

In \cite{SHEM} L.A. Shemetkov possed the problem to describe all formations $\mathfrak{F}=(G\,|\,G=\mathrm{Z}_\mathfrak{F}(G))$. This class of formations contains saturated (local) and solubly saturated (composition or Baer-local) formations.
We shall call formations from this class \emph{$Z$-saturated.}
In \cite{bb1} it was shown that for a formation $\mathfrak{F}$ the class $Z\mathfrak{F}$ is a formation and $\mathfrak{F}\subseteq Z\mathfrak{F}\subseteq\textbf{E}_\Phi\mathfrak{F}$. From this it is straightforward to check that $Z(Z\mathfrak{F})=Z\mathfrak{F}$ and $Z$ is a closure operation on classes of groups.

\begin{proposition}\label{l5} Let $\mathfrak{F}$ be a hereditary formation and  $H\leq G$. Then $\mathrm{Z}_\mathfrak{F}(G)\cap H\leq\mathrm{Z}_\mathfrak{F}(H)$ and $\mathrm{Z}_\mathfrak{F}(G)=\mathrm{Z}_\mathfrak{F}(\mathrm{Z}_\mathfrak{F}(G))$.
\end{proposition}

A subgroup $U$ of  $G$ is called $\mathfrak{X}$-\emph{maximal} in $G$
provided that $(a)$ $U\in\mathfrak{X}$, and $(b)$ if $U\leq V \leq G$ and $V\in\mathfrak{X}$, then $U = V$ \cite[p. 288]{s8}. The symbol $\mathrm{Int}_\mathfrak{X}(G)$
denotes the intersection of all $\mathfrak{X}$-maximal subgroups of $G$ \cite{h4}.

\begin{corollary}\label{l5.1}
  Let $\mathfrak{F}$ be a hereditary  $Z$-saturated formation and  $H$ be an     $\mathfrak{F}$-subgroup of $G$. Then $H\mathrm{Z}_\mathfrak{F}(G)\in\mathfrak{F}$ and $\mathrm{Z}_\mathfrak{F}(G)\leq\mathrm{Int}_\mathfrak{F}(G)$.
\end{corollary}

Proposition \ref{l5} and its corollary are well known for saturated formations. But in that case their proves are based on the properties of the canonical local definition. That is why these results require a new method of a proof for $Z$-saturated formations. Now according to the following result we may assume that $\mathfrak{F}$ is a $Z$-saturated formation in the proof of Theorem \ref{gb}.

\begin{proposition}\label{p9}
Let  $\mathfrak{F}$  be a formation. Then $\mathrm{Z}_{Z\mathfrak{F}}(G)=\mathrm{Z}_\mathfrak{F}(G)$  for every group  $G$ and a subgroup   $K$-$\mathfrak{F}$-subnormal iff it is $K$-$Z\mathfrak{F}$-subnormal.
\end{proposition}

The next step in the proof of Theorem \ref{gb} is to characterize the intersections  $S_\mathfrak{F}(G)$ and $C_\mathfrak{F}(G)$ of all weak $K$-$\mathfrak{F}$-subnormalizers of all Sylow and all cyclic primary subgroups of $G$ respectively.

\begin{proposition}\label{p1} Let $\mathfrak{F}$ be a hereditary formation.

 $(1)$ $S_\mathfrak{F}(G)$ is the largest subgroup among normal subgroups $N$ of $G$ with  $P$ $K$-$\mathfrak{F}$-$sn\,PN$ for every Sylow subgroup $P$ of $G$.

  $(2)$ $C_\mathfrak{F}(G)$ is the largest subgroup among normal subgroups $N$ of $G$ with   $C$ $K$-$\mathfrak{F}$-$sn\,CN$ for every cyclic primary subgroup $C$ of $G$.
 \end{proposition}

 Let   $\mathfrak{F}$ be a hereditary formation. In \cite{vF, wF} the classes of groups $\overline{w}\mathfrak{F}$ and $v^*\mathfrak{F}$ all whose Sylow and cyclic primary subgroups respectively are  $K$-$\mathfrak{F}$-subnormal were studied. In these papers the following results were proved.

 \begin{proposition}\label{wv}
  Let $\mathfrak{F}$ be a hereditary formation.

   $(1)$  $\overline{w}\mathfrak{F}$ and $v^*\mathfrak{F}$ are hereditary formations.

   $(2)$ $\mathfrak{N}\subseteq\overline{w}\mathfrak{F}$ and $\mathfrak{N}\subseteq v^*\mathfrak{F}$.

   $(3)$ $\mathfrak{F}\subseteq\overline{w}\mathfrak{F}$ and $\mathfrak{F}\subseteq v^*\mathfrak{F}$.
 \end{proposition}

The connections between the previous steps are shown in the following proposition:

\begin{proposition}\label{p2} Let $\mathfrak{F}$ be a hereditary formation and  $G$ be a group.

 $(1)$   $C_\mathfrak{F}(G)=\mathrm{Int}_{v^*\mathfrak{F}}(G)$.

$(2)$   $S_\mathfrak{F}(G)=\mathrm{Int}_{\overline{w}\mathfrak{F}}(G)$.

$(3)$    $S_\mathfrak{F}(G)\leq C_\mathfrak{F}(G)$.

$(4)$ $\overline{w}\mathfrak{F}$ and $v^*\mathfrak{F}$  are   $Z$-saturated.
 \end{proposition}

Recall \cite{VM} that a Schmidt $(p, q)$-group
is a Schmidt group with a normal Sylow $p$-subgroup. An \emph{$N$-critical graph} $\Gamma_{Nc}(G)$ of a
 group $G$ \cite[Definition 1.3]{VM} is a directed graph on the vertex set $\pi(G)$ of all prime divisors
 of $|G|$ and $(p, q)$ is an edge of   $\Gamma_{Nc}(G)$ iff $G$ has  a Schmidt $(p, q)$-subgroup.
 An \emph{$N$-critical graph} $\Gamma_{Nc}(\mathfrak{X})$ of a class of groups $\mathfrak{X}$
 \cite[Definition 3.1]{VM} is a directed graph on the vertex set $\pi(\mathfrak{X})=\cup_{G\in\mathfrak{X}}\pi(G)$
 such that $\Gamma_{Nc}(\mathfrak{X})=\cup_{G\in\mathfrak{X}}\Gamma_{Nc}(G)$.

\begin{proposition}[{\cite[Theorem 5.4]{VM}}]\label{5.4}
Let $\sigma=\{\pi_i\mid i\in I\}$ be a partition of the vertex set $V(\Gamma_{Nc}(\mathfrak{X}))$ such that for $i\neq j$ there are no edges between $\pi_i$ and $\pi_j$. Then every $\mathfrak{X}$-group is
the direct product of its Hall $\pi_k$-subgroups, where $k\in \{i \in I \mid \pi(G) \cap \pi_k\neq\emptyset\}$.
\end{proposition}

It is important to note that the main idea of the proves of Theorems \ref{t1.0} and \ref{t1.1} is

\begin{proposition}[{\cite[1, Theorem 2.8(ii)]{s5}}]\label{pr0}
   Let $\mathfrak{F}$ be a saturated formation. If $\mathrm{F}^*(G)\leq \mathrm{Z}_\mathfrak{F}(G)$, then $G\in\mathfrak{F}$.
\end{proposition}

\section{Proves}

 Let $\mathfrak{F}$ be a saturated formation. Recall \cite[p. 95]{f4} that the intersection of all maximal subgroups   $M$ of $G$ with $G/\mathrm{Core}_GM\not\in\mathfrak{F}$ is denoted by $\Delta_\mathfrak{F}(G)$. If $G$ does not have such subgroups, then   $\Delta_\mathfrak{F}(G)=G$.

\begin{lemma}[\hspace{-0.4mm}{\cite[p. 96]{f4}}]\label{delt}
  Let $\mathfrak{F}$ be a saturated formation. Then  $\Delta_\mathfrak{F}(G)/\Phi(G)=\mathrm{Z}_\mathfrak{F}(G/\Phi(G))$.
\end{lemma}

\begin{proof}[Proof of Theorem \ref{t1.0}] $(1)$
Suppose that $\mathfrak{F}$-contains  every group $G$ with all maximal subgroups $\tilde{\mathrm{F}}(G)$-$K$-$\mathfrak{F}$-subnormal.

 Since every maximal subgroup of a nilpotent group $G$ is normal, it is  $\tilde{\mathrm{F}}(G)$-$K$-$\mathfrak{F}$-subnormal. Hence $G\in\mathfrak{F}$. So $\mathfrak{N}\subseteq\mathfrak{F}$.  From $\Phi(G)\subseteq\mathrm{\tilde{F}}(G)$ it follows that if $G/\Phi(G)\in\mathfrak{F}$, then all maximal subgroups of $G$ are    $\tilde{\mathrm{F}}(G)$-$K$-$\mathfrak{F}$-subnormal. Hence $G\in\mathfrak{F}$. Thus $\mathfrak{F}$ is saturated.

Suppose that $\mathfrak{F}$ is saturated and  $\mathfrak{N}\subseteq\mathfrak{F}$.

Assume now that every maximal subgroup of $G$ is $\tilde{\mathrm{F}}(G)$-$K$-$\mathfrak{F}$-subnormal.
Let  $M$ be a maximal subgroup of $G$ with  $M\tilde{\mathrm{F}}(G)=G$. Then $M\trianglelefteq G$ or  $G/\mathrm{Core}_GM\in\mathfrak{F}$. From $\mathfrak{N}\subseteq\mathfrak{F}$ it follows that  $Z_p\in\mathfrak{F}$ for all $p\in\mathbb{P}$. It means that if $M\trianglelefteq G$, then $G/\mathrm{Core}_GM\simeq Z_p\in\mathfrak{F}$.

   Now $\tilde{\mathrm{F}}(G)\leq \Delta_\mathfrak{F}(G)$. So $\tilde{\mathrm{F}}(G/\Phi(G))=\tilde{\mathrm{F}}(G)/\Phi(G)
\leq\Delta_\mathfrak{F}(G)/\Phi(G)=\mathrm{Z}_\mathfrak{F}(G/\Phi(G))$ by Lemma \ref{delt}. Therefore $G/\Phi(G)\in\mathfrak{F}$ by Proposition \ref{pr0}. Since $\mathfrak{F}$ is saturated, $G\in\mathfrak{F}$.

$(2)$  Suppose that $\mathfrak{F}$-contains  every group $G$ with all maximal subgroups $\mathrm{F}^*(G)$-$K$-$\mathfrak{F}$-sub\-nor\-mal. Let $G$ and $H$ be a group and a simple non-abelian group respectively. Consider $W=H\wr_{reg} G$. It is known that in this case the base $B$  of $W$ is a unique minimal normal non-abelian subgroup of $W$.

From  \cite{20} it follows that there exists a faithful  $\mathbb{F}_pW$-module $A$ for a $p\in\pi(W)$,
  such that  $A\rightarrow E\twoheadrightarrow W$ where $A\stackrel {W}{\simeq} \Phi(E)$ and
 $E/\Phi(E)\simeq W$. Note that $\mathrm{F}^*(E)=\Phi(E)$. Hence all maximal subgroups of $E$ are  $\mathrm{F}^*(E)$-$K$-$\mathfrak{F}$-subnormal. Therefore $E\in\mathfrak{F}$. Thus $G\in\mathfrak{F}$ as a quotient group of $E$. So $\mathfrak{G}\subseteq\mathfrak{F}$. Thus $\mathfrak{G}=\mathfrak{F}$.
\end{proof}

\begin{proof}[Proof of Corollary \ref{kr}]
Let  $M$ be a maximal subgroup of $G$.   If $M\cap\mathrm{F}(G)$ is a maximal subgroup of $\mathrm{F}(G)$, then $\mathrm{F}(G)/(M\cap\mathrm{F}(G))\simeq Z_p$ for some  $p\in\mathbb{P}$. Hence $|G:M|=p$. Since  $G$ is soluble, we see that $M$ is $K$-$\mathfrak{U}$-subnormal by  \cite[Lemma 3.4]{gs3} and $\mathrm{F}(G)=\tilde{\mathrm{F}}(G)$. Thus all maximal subgroups of $G$ are $\tilde{\mathrm{F}}(G)$-$K$-$\mathfrak{U}$-subnormal. Hence $G$ is supersoluble by Theorem \ref{t1.0}.
\end{proof}

\begin{proof}[Proof of Proposition \ref{l5}]  Let $H\leq G$ and
$1=Z_0\trianglelefteq Z_1\trianglelefteq\dots\trianglelefteq Z_n=\mathrm{Z}_\mathfrak{F}(G)$ be a $G$-composition series of $\mathrm{Z}_\mathfrak{F}(G)$. Then
\[1=Z_0\cap H\trianglelefteq Z_1\cap H\trianglelefteq\dots\trianglelefteq Z_n\cap H=\mathrm{Z}_\mathfrak{F}(G)\cap H\]
is a part of normal series of   $H$. Let $Z_{i-1}\leq K\leq T\leq Z_i$ and $T/K$ be a chief factor of  $H$. Then
\[L=(Z_i/Z_{i-1})\rtimes (HC_G(Z_i/Z_{i-1})/C_G(Z_i/Z_{i-1}))\leq (Z_i/Z_{i-1})\rtimes (G/C_G(Z_i/Z_{i-1}))\in\mathfrak{F}.\]
Since $\mathfrak{F}$ is a hereditary formation, we see that  $$HC_G(Z_i/Z_{i-1})/C_G(Z_i/Z_{i-1})\simeq H/C_H(Z_i/Z_{i-1})\in\mathfrak{F}.$$ Note that $C_H(Z_i/Z_{i-1})\leq C_H(T/K)$. Then  $H/C_H(T/K)\in\mathfrak{F}$  as a quotient group of an  $\mathfrak{F}$-group $H/C_H(Z_i/Z_{i-1})$.

Assume that $Z_i/Z_{i-1}$ is non-abelian. Then $H/C_H(T/K)$ is a primitive group of type 2. It is known that in this case $(T/K)\rtimes (H/C_H(T/K))$ is a primitive group of type 3 and its quotient group by any of its minimal normal subgroups is isomorphic to  $H/C_H(T/K)$. Since $\mathfrak{F}$ is a formation, we see that  $(T/K)\rtimes (H/C_H(T/K))\in\mathfrak{F}$.

Assume that $Z_i/Z_{i-1}$ is abelian. Note that $L\in\mathfrak{F}$ and isomorphic groups $$HC_G(Z_i/Z_{i-1})/C_G(Z_i/Z_{i-1})  \textrm{ and }  H/C_H(Z_i/Z_{i-1})$$  acts (by conjugation) in the same way on  $Z_i/Z_{i-1}$. Now $L\simeq (Z_i/Z_{i-1})\rtimes (H/C_H(Z_i/Z_{i-1})).$ From $K\trianglelefteq H$ it follows that
   $M=(Z_i/K)\rtimes (H/C_H(Z_i/Z_{i-1}))\in\mathfrak{F}.$

   Note that  $Z_i/K$ acts (by conjugation) trivially on $T/K$. Thus $M$ acts on $T/K$ in the same way as  $H/C_H(Z_i/Z_{i-1})$. From
   $$(H/C_H(Z_i/Z_{i-1}))/C_{H/C_H(Z_i/Z_{i-1})}(T/K)\simeq H/C_H(T/K)$$
   and   \cite[Corollary 2.2.5]{s9} for $M$ it follows that  $$(T/K)\rtimes (H/C_H(T/K))\in\mathfrak{F}.$$ Thus $\mathrm{Z}_\mathfrak{F}(G)\cap H\leq\mathrm{Z}_\mathfrak{F}(H)$.

Note that $\mathrm{Z}_\mathfrak{F}(G)=\mathrm{Z}_\mathfrak{F}(G)\cap \mathrm{Z}_\mathfrak{F}(G)\leq \mathrm{Z}_\mathfrak{F}(\mathrm{Z}_\mathfrak{F}(G))\leq \mathrm{Z}_\mathfrak{F}(G)$.
Thus $\mathrm{Z}_\mathfrak{F}(G)=\mathrm{Z}_\mathfrak{F}(\mathrm{Z}_\mathfrak{F}(G))$.
 \end{proof}

\begin{proof}[Proof of  Corollary \ref{l5.1}]
  From Proposition \ref{l5} it follows that $\mathrm{Z}_\mathfrak{F}(G)\leq \mathrm{Z}_\mathfrak{F}(H\mathrm{Z}_\mathfrak{F}(G))$. Since \linebreak $H\mathrm{Z}_\mathfrak{F}(G)/\mathrm{Z}_\mathfrak{F}(G)\in\mathfrak{F}$, we see that   $H\mathrm{Z}_\mathfrak{F}(G)/\mathrm{Z}_\mathfrak{F}(H\mathrm{Z}_\mathfrak{F}(G))\in\mathfrak{F}$. Hence $H\mathrm{Z}_\mathfrak{F}(G)\in\mathfrak{F}$.

 Let $M$ be an $\mathfrak{F}$-maximal subgroup of  $G$. Then $M\mathrm{Z}_\mathfrak{F}(G)\in\mathfrak{F}$. So $M\mathrm{Z}_\mathfrak{F}(G)=M$. Thus  $\mathrm{Z}_\mathfrak{F}(G)\leq\mathrm{Int}_\mathfrak{F}(G)$.
\end{proof}

\begin{proof}[Proof of  Proposition \ref{p9}]
 Note that the $\mathfrak{X}$-hypercenter and $K$-$\mathfrak{X}$-subnormality are defined by the set of all primitive   $\mathfrak{X}$-groups.   According to \cite{bb1}   $\mathfrak{F}\subseteq Z\mathfrak{F}\subseteq \textbf{E}_\Phi\mathfrak{F}$. Thus the sets of all primitive $\mathfrak{F}$-groups and $Z\mathfrak{F}$-groups coincide.
\end{proof}

 \begin{proof}[Proof of  Proposition \ref{p1}]
   $(1)$ Let $N\trianglelefteq G$ with  $P$ $K$-$\mathfrak{F}$-$sn\,PN$ for every Sylow subgroup   $P$ of $G$. If $S$ is a weak  $K$-$\mathfrak{F}$-subnormalizer of  $P$ in $G$, then $PN$  $K$-$\mathfrak{F}$-$sn\, SN$  by Lemma \ref{lemN}. Hence $P$  $K$-$\mathfrak{F}$-$sn\, SN$ by   $(3)$ of Lemma \ref{l3.1}.  Now $SN=S$ by the definition of a weak   $K$-$\mathfrak{F}$-subnormalizer. Thus $N\leq S_\mathfrak{F}(G)$.

   From the other hand, since  $\mathfrak{F}$ is a hereditary formation and $PS_\mathfrak{F}(G)$ lies in every weak $K$-$\mathfrak{F}$-subnormalizer of every Sylow subgroup $P$ of $G$, we see that $P$ $K$-$\mathfrak{F}$-$sn\,PS_\mathfrak{F}(G)$ for every Sylow subgroup   $P$ of $G$ by Lemma \ref{l3.2}.  Thus $S_\mathfrak{F}(G)$ is the largest normal subgroup  $N$ of $G$ with $P$ $K$-$\mathfrak{F}$-$sn\,PN$ for every Sylow subgroup $P$ of $G$.

  The proof of $(2)$ is the same. \end{proof}

\begin{proof}[Proof of Proposition \ref{p2}]
 $(1)$  Note that   $\mathfrak{N}\subseteq v^*\mathfrak{F}$ and $v^*\mathfrak{F}$ is a hereditary formation by Proposition~\ref{wv}. Hence $P\mathrm{Int}_{v^*\mathfrak{F}}(G)\in v^*\mathfrak{F}$ for every cyclic primary subgroup   $P$ of $G$.  Therefore $P$ $K$-$\mathfrak{F}$-$sn\,P\mathrm{Int}_{v^*\mathfrak{F}}(G)$  for every cyclic primary subgroup   $P$ of $G$.  Thus $\mathrm{Int}_{v^*\mathfrak{F}}(G)\leq C_\mathfrak{F}(G)$  by $(2)$ of Proposition  \ref{p1}.

From the other hand let  $M$ be  a $v^*\mathfrak{F}$-maximal subgroup of  $G$ and $P$ be a cyclic primary subgroup of  $MC_\mathfrak{F}(G)$.
Since $MC_\mathfrak{F}(G)/C_\mathfrak{F}(G)\in v^*\mathfrak{F}$, we see thta $PC_\mathfrak{F}(G)/C_\mathfrak{F}(G)$    $K$-$\mathfrak{F}$-$sn\,MC_\mathfrak{F}(G)/C_\mathfrak{F}(G)$. Hence $PC_\mathfrak{F}(G)$    $K$-$\mathfrak{F}$-$sn\,MC_\mathfrak{F}(G)$ by $(2)$  of Lemma  \ref{l3.1}. Note that   $P$   $K$-$\mathfrak{F}$-$sn\,PC_\mathfrak{F}(G)$ by Proposition  \ref{p1}. So $P$   $K$-$\mathfrak{F}$-$sn\,MC_\mathfrak{F}(G)$ by $(3)$ of Lemma \ref{l3.1}. Thus $MC_\mathfrak{F}(G)\in v^*\mathfrak{F}$ by the definition of  $v^*\mathfrak{F}$. Hence $MC_\mathfrak{F}(G)=M$. Therefore $C_\mathfrak{F}(G)\leq\mathrm{Int}_{v^*\mathfrak{F}}(G)$. Thus $\mathrm{Int}_{v^*\mathfrak{F}}(G)= C_\mathfrak{F}(G)$.

%Отметим, что согласно теоремам 3.1 и 3.3 из \cite{gs6}, $\mathfrak{N}\subseteq \overline{w}\mathfrak{F}$ и $\overline{w}\mathfrak{F}$ --- наследственная формация.
The proof of  $(2)$ is the same.

$(3)$ Since every cyclic primary subgroup is subnormal in some Sylow subgroup, we see that  $P$ $K$-$\mathfrak{F}$-$sn\,PS_\mathfrak{F}(G)$ for every cyclic primary subgroup $P$ of $G$. So $S_\mathfrak{F}(G)\leq C_\mathfrak{F}(G)$ by Proposition~\ref{p1}.

$(4)$ Assume that   $\overline{w}\mathfrak{F}$ is not a   $Z$-saturated formation. Let chose a minimal order group   $G$ from  $Z(\overline{w}\mathfrak{F})\setminus \overline{w}\mathfrak{F}$. From Proposition \ref{l5} it follows that   $Z\overline{w}\mathfrak{F}$ is a hereditary formation. So  $G$ is $\overline{w}\mathfrak{F}$-critical. Now $|\pi(G)|>1$ by   $(2)$ of Proposition \ref{wv}. From  $\overline{w}\mathfrak{F}\subset Z\overline{w}\mathfrak{F}\subseteq \textbf{E}_\Phi \overline{w}\mathfrak{F}$ it follows that   $\Phi(G)\neq 1$ and $G/\Phi(G)\in \overline{w}\mathfrak{F}$. Let   $P$ be a Sylow subgroup of $G$. Then $P\Phi(G)< G$ and $P\Phi(G)\in \overline{w}\mathfrak{F}$. Hence $P$ $K$-$\mathfrak{F}$-$sn\,P\Phi(G)$. From $G/\Phi(G)\in \overline{w}\mathfrak{F}$ it follows that $P\Phi(G)/\Phi(G)$ $K$-$\mathfrak{F}$-$sn\,G/\Phi(G)$. Therefore $P\Phi(G)$ $K$-$\mathfrak{F}$-$sn\,G$. Thus $P$ $K$-$\mathfrak{F}$-$sn\,G$. It means that  $G\in \overline{w}\mathfrak{F}$, a contradiction. Thus  $\overline{w}\mathfrak{F}$ is a  $Z$-saturated formation. The proof for $v^*\mathfrak{F}$ is the same.
\end{proof}

\begin{proposition}\label{either}
  Let $\mathfrak{F}$ be a hereditary formation such that one of the following claims holds:

  $(1)$  The intersection of all weak  $K$-$\mathfrak{F}$-subnormalizers of all Sylow  subgroups is the  $\mathfrak{F}$-hypercenter.

$(2)$ $\mathfrak{F}$ contains every group  $G$ all whose Sylow subgroups are  $\mathrm{F}^*(G)$-$K$-$\mathfrak{F}$-subnormal.

Then there is a partition   $\sigma$ of $\mathbb{P}$ such that $\mathfrak{F}$ is the class of all  $\sigma$-nilpotent groups.
\end{proposition}

\begin{proof}
$(a)$ \emph{$\mathfrak{N}\subseteq\mathfrak{F}$ is $Z$-saturated}.

Assume that $(1)$ holds. According to Proposition \ref{p9} this statement means the same for   $\mathfrak{F}$ and $Z\mathfrak{F}$. Note that $Z\mathfrak{F}=Z(Z\mathfrak{F})$ by Proposition \ref{p9}. Without lose of generality we may assume that  $\mathfrak{F}$ is $Z$-saturated. Since in every nilpotent group every Sylow subgroup is subnormal and $Z\mathfrak{F}=\mathfrak{F}$ we see that   $\pi(\mathfrak{F})=\mathbb{P}$ and $\mathfrak{N}\subseteq\mathfrak{F}$ by $(1)$.

Assume that $(2)$ holds, i.e.   $\mathfrak{F}$ contains every group   $G$ all whose Sylow subgroups are  $\mathrm{F}^*(G)$-$K$-$\mathfrak{F}$-subnormal. So  $\mathfrak{F}$ contains every group  $G$ all whose Sylow subgroups are   $K$-$\mathfrak{F}$-subnormal. Hence $\mathfrak{F}=\overline{w}\mathfrak{F}$.  Now $\mathfrak{N}\subseteq\mathfrak{F}$ by Proposition \ref{wv} and  $\mathfrak{F}$ is $Z$-saturated by (4) of Proposition \ref{p2}.

 $(b)$ \emph{Assume $L$ is faithful irreducible $\mathbb{F}_pG$-module,    $T=L\rtimes G$  and $L\leq S_\mathfrak{F}(T)$. Then $G\in\mathfrak{F}$.}

Assume that $(1)$ holds. Now $L\leq S_\mathfrak{F}(G)=\mathrm{Z}_{\mathfrak{F}}(T)$. Hence $L\rtimes (T/C_T(L))\in\mathfrak{F}$. Thus $G\simeq T/C_T(L)\in\mathfrak{F}$, the contradiction.

Assume that $(2)$ holds. So $L=\mathrm{F}^*(T)\leq S_\mathfrak{F}(T)$. Now $T\in\mathfrak{F}$ by $(2)$. Thus $G\in\mathfrak{F}$ as a quotient group of $T$, the contradiction.

$(c)$ \emph{Let $\pi(p)=\{q\in\mathbb{P}\,|\,(p, q)\in\Gamma_{Nc}(\mathfrak{F})\}\cup\{p\}$. Then $\mathfrak{F}$ contains every $q$-closed $\{p, q\}$-group for every  $q\in\pi(p)$.}

Assume the contrary. Let $G$ be a minimal order counterexample. Since $\mathfrak{F}$ and the class of all $q$-closed groups are hereditary formations, we see that  $G$ is an $\mathfrak{F}$-critical group, $G$  has a unique minimal normal subgroup $N$ and $G/N\in\mathfrak{F}$. Let $P$ be a Sylow $p$-subgroup of $G$.   If $NP<G$, then $NP\in\mathfrak{F}$. Hence $P$ $K$-$\mathfrak{F}$-$sn\,PN$  and $PN/N$ $K$-$\mathfrak{F}$-$sn\,G/N$. From Lemma \ref{l3.1} it follows that $P$ $K$-$\mathfrak{F}$-$sn\,G$. Since $G$  is a $q$-closed $\{p, q\}$-group, we see that every Sylow subgroup of $G$ is $K$-$\mathfrak{F}$-subnormal. If $(1)$ or $(2)$ hold, then $G\in Z\mathfrak{F}=\mathfrak{F}$ or $G\in\mathfrak{F}$ respectively, a contradiction.

Note that $N$ is a Sylow $q$-subgroup  and $\mathrm{O}_p(G)=1$. By Lemma \ref{10.3B} there exists a faithful irreducible $\mathbb{F}_pG$-module $L$. Let $T=L\rtimes G$.  Therefore for every chief factor $H/K$ of $NL$  a group    $(H/K)\rtimes C_{NL}(H/K)$ is isomorphic to one of the following groups  $Z_p, Z_q$ and a Schmidt $(p, q)$-group with the trivial Frattini subgroup. Note that all these groups belong $\mathfrak{F}$. So  $NL\in Z\mathfrak{F}=\mathfrak{F}$. Note that $L\leq \mathrm{O}_p(T)$. Hence  $L\leq S_\mathfrak{F}(T)$ by Proposition \ref{p1}.
Thus $G\in\mathfrak{F}$ by $(b)$, a contradiction.

From $(c)$ it follows that

 $(d)$ \emph{$\Gamma_{Nc}(\mathfrak{F})$ is undirected, i.e $(p, q)\in\Gamma_{Nc}(\mathfrak{F})$ iff $(q, p)\in\Gamma_{Nc}(\mathfrak{F})$.}

$(e)$ \emph{Let   $p, q$ and $r$ be different primes.  If $(p, r), (q, r)\in\Gamma_{Nc}(\mathfrak{F})$, then   $(p, q)\in\Gamma_{Nc}(\mathfrak{F})$}.

There exists a faithful irreducible    $\mathbb{F}_pZ_q$-module $P$ by Lemma  \ref{10.3B}. Let $G=P\rtimes Z_q$.  Then there exists a faithful irreducible     $\mathbb{F}_rG$-module $R$ by Lemma  \ref{10.3B}. Let $T=R\rtimes G$. From $(c)$ it follows that   $\mathfrak{F}$-contains  all $r$-closed $\{p, r\}$-groups and $\{q, r\}$-groups. Thus $R\leq S_\mathfrak{F}(T)$  by Proposition~\ref{p1}. Thus $G\in\mathfrak{F}$ by $(b)$. Note that $G$  is a Schmidt $(p, q)$-group.

$(f)$   \emph{$\mathfrak{F}=\underset{i\in I}\times\mathfrak{G}_{\pi_i}$ for some partition $\sigma$ of $\mathbb{P}$, i.e.\ $\mathfrak{F}$ is the class of all $\sigma$-nilpotent groups.  }

From $(d)$ and $(e)$ it follows that $\Gamma_{Nc}(\mathfrak{F})$ is a disjoint union of  complete (directed) graphs $\Gamma_i$, $i\in I$. Let $\pi_i=V(\Gamma_i)$. Then $\sigma=\{\pi_i\,|\,i\in I\}$ is a partition of $\mathbb{P}$. From Proposition \ref{5.4} it follows that $\mathfrak{F}\subseteq\underset{i\in I}\times\mathfrak{G}_{\pi_i}$.

Let show that
$\mathfrak{G}_{\pi_i}\subseteq\mathfrak{F}$ for every $p$.
It is true  if $|\pi_i|=1$. Assume now $|\pi_i|>1$.  Suppose the contrary and let a group  $G$ be a minimal order group from $\mathfrak{G}_{\pi_i}\setminus\mathfrak{F}$. Then $G$ has a unique minimal normal subgroup, $\pi(G)\subseteq\pi_i$ and $|\pi(G)|>1$. Note that $\mathrm{O}_q(G)=1$ for some $q\in\pi(G)$. Hence there exists faithful irreducible    $\mathbb{F}_qG$-module $N$ by Lemma \ref{10.3B}. Let $T=N\rtimes G$. Hence $NP\in\mathfrak{F}$ for every Sylow subgroup $P$ of  $T$ by $(c)$. Now $N\leq S_\mathfrak{F}(T)$ by Proposition \ref{p1}.
So $G\in\mathfrak{F}$ by $(b)$, the contradiction.

It means that $\underset{i\in I}\times\mathfrak{G}_{\pi_i}\subseteq\mathfrak{F}$.
Thus $\mathfrak{F}=\underset{i\in I}\times\mathfrak{G}_{\pi_i}$, i.e.\ $\mathfrak{F}$ is the class of all $\sigma$-nilpotent groups.
\end{proof}

\begin{proof}[Proof of Theorem \ref{gb}]
$(1)\Rightarrow(2)$. Since $\mathfrak{F}\subseteq \overline{w}\mathfrak{F}$ by Proposition \ref{wv}, we see that $\mathrm{Z}_\mathfrak{F}(G)\leq\mathrm{Z}_{\overline{w}\mathfrak{F}}(G)$ for every group $G$. Note that $\mathrm{Z}_{\overline{w}\mathfrak{F}}(G)\leq\mathrm{Int}_{\overline{w}\mathfrak{F}}(G)$  for every group $G$ by Corollary \ref{l5.1} and $(4)$ of Proposition \ref{p2}. According to Proposition \ref{p2},   $S_\mathfrak{F}(G)=\mathrm{Int}_{\overline{w}\mathfrak{F}}(G)$ and
  $S_\mathfrak{F}(G)\leq C_\mathfrak{F}(G)$ for every group $G$. From these and   $(1)$ it follows that
  $$\mathrm{Z}_\mathfrak{F}(G)\leq\mathrm{Z}_{\overline{w}\mathfrak{F}}(G)\leq\mathrm{Int}_{\overline{w}\mathfrak{F}}(G)
  =S_\mathfrak{F}(G)\leq C_\mathfrak{F}(G)=\mathrm{Z}_\mathfrak{F}(G)$$
 for every group $G$.  Thus $\mathrm{Z}_\mathfrak{F}(G) =S_\mathfrak{F}(G)$ for every group $G$.

$(2)\Rightarrow(3)$. Directly follows from Proposition \ref{either}.

$(3)\Rightarrow(1)$. Assume that there is a partition $\sigma=\{\pi_i\,|\,i\in I\}$ of $\mathbb{P}$ with $\mathfrak{F}=\times_{i\in I}\mathfrak{G}_{\pi_i}$. Then $\mathfrak{F}$ is a lattice formation. According to \cite[Theorem B and Corollary E.2]{vF} $v^*\mathfrak{F}=\mathfrak{F}$. By \cite[Theorem A and Proposition 4.2]{h4}    $\mathrm{Int}_\mathfrak{F}(G)=\mathrm{Z}_\mathfrak{F}(G)$ holds for every group $G$. By (2) of Proposition \ref{p2}, $C_\mathfrak{F}(G)=\mathrm{Int}_{v^*\mathfrak{F}}(G)$ for every group $G$. Thus
$$C_\mathfrak{F}(G)=\mathrm{Int}_{v^*\mathfrak{F}}(G)=\mathrm{Int}_{\mathfrak{F}}(G)=\mathrm{Z}_\mathfrak{F}(G) $$ for every group $G$.  \end{proof}

\begin{proof}[Proof of Theorem \ref{t1.1}]
$(1)\Rightarrow(2)$. Note that every cyclic primary subgroup is subnormal in some Sylow subgroup. Hence if all Sylow subgroups of   $G$ are $\mathrm{F}^*(G)$-$K$-$\mathfrak{F}$-subnormal, then all cyclic primary subgroups of   $G$ are also $\mathrm{F}^*(G)$-$K$-$\mathfrak{F}$-subnormal. Thus $G\in\mathfrak{F}$.

  $(2)\Rightarrow(3)$. Directly follows from Proposition \ref{either}.

%$(3)\Rightarrow(1)$. Assume that there is a partition $\sigma=\{\pi_i\,|\,i\in I\}$ of $\mathbb{P}$ with $\mathfrak{F}=\times_{i\in I}\mathfrak{G}_{\pi_i}$. Then $\mathfrak{F}$ is a lattice formation. According to \cite[Theorem B and Corollary E.2]{gs5} $v^*\mathfrak{F}=\mathfrak{F}$. By \cite{h4} or ??  $\mathrm{Int}_\mathfrak{F}(G)=\mathrm{Z}_\mathfrak{F}(G)$ holds for every group $G$. By (2) of Proposition \ref{p2}, $C_\mathfrak{F}(G)=\mathrm{Int}_{v^*\mathfrak{F}}(G)$ for every group $G$. Thus
%$$C_\mathfrak{F}(G)=\mathrm{Int}_{v^*\mathfrak{F}}(G)=\mathrm{Int}_{\mathfrak{F}}(G)=\mathrm{Z}_\mathfrak{F}(G) $$ for every group $G$.

 \smallskip

  $(3)\Rightarrow(1)$. Assume that all cyclic primary subgroups of  $G$ are $\mathrm{F}^*(G)$-$K$-$\mathfrak{F}$-subnormal. Then $\mathrm{F}^*(G)\leq\mathrm{Z}_\mathfrak{F}(G)$  by Proposition   \ref{p1}  and Theorem  \ref{gb}.  Now  $G\in\mathfrak{F}$ by Proposition \ref{pr0}.
\end{proof}

\begin{proof}[Proof of Theorem \ref{new}] $(1)\Rightarrow (2).$  Assume that $G=AB$ where all Sylow subgroups of $A$ and $B$ are  $\mathrm{F}^*(G)$-$K$-$\mathfrak{F}$-subnormal. Since every cyclic primary subgroups $C$ is subnormal in some Sylow subgroup $P$ of $A$, we see that    $C\trianglelefteq\trianglelefteq P$ $K$-$\mathfrak{F}$-$sn\,P\mathrm{F}^*(G)$. Now $C$ $K$-$\mathfrak{F}$-$sn\,C\mathrm{F}^*(G)$ by  Lemma \ref{l3.1}. Hence  $C$ is  $\mathrm{F}^*(G)$-$K$-$\mathfrak{F}$-subnormal. Thus all cyclic primary subgroups of  $A$ are   $\mathrm{F}^*(G)$-$K$-$\mathfrak{F}$-subnormal.  We can prove the same statement for $B$. Now $G\in\mathfrak{F}$ by  $(1)$.

$(2)\Rightarrow (3).$ From $G=GG$ and $(2)$ it follows that $\mathfrak{F}$ contains every group $G$ all whose Sylow subgroups are $\mathrm{F}^*(G)$-$K$-$\mathfrak{F}$-subnormal. Thus
there is a partition   $\sigma$ of $\mathbb{P}$ such that $\mathfrak{F}=\mathfrak{N}_\sigma$  by Theorem  \ref{t1.1}.

%Напомним, что формация называется решеточной, если в любой группе порождение $\mathfrak{F}$-субнормальных подгрупп $\mathfrak{F}$-субнормально. Согласно теореме 6.3.15 \cite{BB}, $\mathfrak{F}=\underset{i\in I}\times\mathfrak{G}_{\pi_i}$ является решеточной формацией.

$(3)\Rightarrow(1)$.
Let $G=AB$ where all cyclic primary subgroups of  $A$ and $B$ are   $\mathrm{F}^*(G)$-$K$-$\mathfrak{F}$-subnormal. By  \cite[Lemma 11.6]{f4} there are Sylow $p$-subgroups $P_1$, $P_2$ and $P$ of $A$, $B$ and $G$  respectively with  $P_1P_2=P$.

 Let  $C\leq P_1$  be a cyclic primary subgroup. Since $C$ $K$-$\mathfrak{F}$-$sn\,P_1$, we see that $C\mathrm{F}^*(G)$ $K$-$\mathfrak{F}$-$sn\,P_1\mathrm{F}^*(G)$ by Lemma \ref{lemN}. From $C$ $K$-$\mathfrak{F}$-$sn\,C\mathrm{F}^*(G)$ it follows  that  $C$ $K$-$\mathfrak{F}$-$sn\,P_1\mathrm{F}^*(G)$ by $(3)$ of Lemma \ref{l3.1}.

 Since $\mathfrak{F}$ has the lattice property for $K$-$\mathfrak{F}$-subnormal subgroups by Lemma \ref{lattice} and   $P_1$   is generated by all its cyclic primary subgroups, we see that  $P_1$ $K$-$\mathfrak{F}$-$sn\,P_1\mathrm{F}^*(G)$.

 From  $P_1$ $K$-$\mathfrak{F}$-$sn\,P$ it follows that  $P_1\mathrm{F}^*(G)$ $K$-$\mathfrak{F}$-$sn\,P\mathrm{F}^*(G)$ by Lemma \ref{lemN}. Since $P_1$ $K$-$\mathfrak{F}$-$sn\,P_1\mathrm{F}^*(G)$, we see that  $P_1$ $K$-$\mathfrak{F}$-$sn\,P\mathrm{F}^*(G)$ by   $(3)$ of Lemma \ref{l3.1}. The same argument shows that  $P_2$ $K$-$\mathfrak{F}$-$sn\,P\mathrm{F}^*(G)$. Thus  $P$ $K$-$\mathfrak{F}$-$sn\,P\mathrm{F}^*(G)$ by the lattice property.

 Since all Sylow  $p$-subgroups of $G$ are conjugate, they all are   $\mathrm{F}^*(G)$-$K$-$\mathfrak{F}$-subnormal. By analogy one can show that all Sylow subgroups of   $G$  are $\mathrm{F}^*(G)$-$K$-$\mathfrak{F}$-subnormal. Now $G\in\mathfrak{F}$ by Theorem \ref{t1.1}.
   \end{proof}

\end{document}